\documentclass[12pt]{article}%
\usepackage{graphicx}
\usepackage{mathrsfs}
\usepackage{amsmath}
\usepackage{amsfonts}
\usepackage[T2A]{fontenc}
\usepackage[cp1251]{inputenc}
\usepackage[russian,english]{babel}
\usepackage{amssymb}%
\usepackage[table]{xcolor}
\usepackage{ctable}
\newtheorem{theorem}{Theorem}

\newtheorem{exam}[theorem]{\bf Example}

\newtheorem{claim}[theorem]{Claim}

\newtheorem{lemma}[theorem]{Lemma}

\newtheorem{proposition}[theorem]{Proposition}
\newtheorem{remark}[theorem]{Remark}

\newenvironment{proof}[1][Proof]{\textbf{#1.} }{\ \rule{0.5em}{0.5em}}
\usepackage{setspace}
\begin{document}
\author {Martial Longla\\ University of Mississippi,\\University, MS, USA.\\ longla$_{-}$m$_{-}$martial@yahoo.com}
\title{On Dependence Structure of Copula-based Markov chains}
\maketitle

\abstract{We consider dependence coefficients for stationary Markov chains. We emphasize on some equivalencies for reversible Markov chains. We improve some known results and provide a necessary condition for Markov chains based on Archimedean copulas to be exponential $\rho$-mixing. We analyse the example of the Mardia and Frechet copula families using small sets.}



\section{Introduction}
This work is motivated by applications in Bayesian analysis of Monte Carlo Markov chains. Longla and Peligrad (2012), Longla (2013) have provided several theorems on exponential $\rho$-mixing and geometric ergodicity of convex combinations of geometrically ergodic Markov chains. This work completes the ideas provided in the two cited papers, that one can read for more information on copulas and their importance in assessing the dependence structure of Markov chains.

\subsection{Notations}
In this paper we use the following notations: 
$\displaystyle L^{p}(0,1)=\big\{g(x): \int_{0}^{1}|g|^{p}(x)dx<\infty\big\},$

 $\displaystyle L^{p}_{0}(0,1)=\big\{g(x): \int_{0}^{1}|g|^{p}(x)dx<\infty, \int_{0}^{1}g(x)dx=0\big\},$  $\displaystyle ||g||_{p}=\Big(\int_{0}^{1}|g|^{p}(x)dx\Big)^{1/p}.$ 

 For $i=1,2$, $\displaystyle A_{,i}(x_{1},x_{2})=\frac{\partial A(x_{1},x_{2})}{\partial x_{i}}$ and $c(x,y)=C_{,12}(x,y)$ is called density of the copula $C(x,y)$. $\mathcal{R}$ is the Borel $\sigma$-algebra. $A^{c}$ is the complement of $A$. $\mu$ is the Lebesgue measure on $[0,1]$. $I$ stands for the interval $[0,1]$ and  $[x]$ is the integer part of the number $x$. $\mathbb{R}^{+}$ is the set of positive real numbers. $\mathbb{N}$ is the set of natural numbers.

\subsection{Definitions}
 A 2-copula is a bivariate function $C:[0,1]\times [0, 1] \rightarrow [0, 1]=I$, such that $C(0, x)=C(x, 0)=0$ (meaning that C is grounded), $C(1, x)= C(x, 1) = x $  for all $x\in[0,1]$ ( meaning that each coordinate is uniform on $I$),  for all $[x_{1}, x_{2}]\times [y_{1}, y_{2}]\subset I^{2}$, $C(x_{1}, y_{1})+C(x_{2}, y_{2})-C(x_{1}, y_{2})-C(x_{2}, y_{1})\geq 0.$ Therefore, any convex combination of 2-copulas is a 2-copula. The increased interest in the theory of copulas and its application is due to the following fact. If $X_{1}, X_{2}$ are random variables with joint distribution $F$ and marginal distributions $F_{1},F_{2} $, then the function $C$ defined via $C(F_{1}(x_{1}), F_{2}(x_{2}))=F(x_{1}, x_{2})$ is a copula (this is the Sklar's theorem). 

\noindent $\displaystyle A*B(x,y)=\int_{0}^{1}A_{,2}(x,t)B_{,1}(t,y)dt$ is a copula, fold product of the copulas $A(x,y)$ and $B(x,y)$.
 
\noindent Some popular examples of copulas are: the Hoeffding upper bound $M(u,v)=\min(u,v)$, the Hoeffding lower bound $W(u,v)=\max{(u+v-1,0)}$ and the independence copula $P(u,v)=uv$. Any copula has its graph between the graphs of $W$ and $M$. $P$ is the copula associated to two independent random variables. Another popular class of copulas is the Archimedean family of copulas. Given a decreasing concave up (convex) function $\varphi: [0,1]\rightarrow [0,\infty)$ such that $\varphi(1)=0$.

 If $\varphi(0)=\infty$, then  $\varphi$ is called a strict generator or generator of the strict Archimedean copula 

$\displaystyle
C(u,v)=\varphi^{-1}(\varphi(u)+\varphi(v)) , \quad\mbox{with}  \quad c(u,v)=-\frac{\varphi^{''}{ o}\varphi^{-1}(\varphi(u)+\varphi(v))\varphi^{'}(u)\varphi^{'}(v)}{\Big(\varphi^{'}{ o}\varphi^{-1}(\varphi(u)+\varphi(v))\Big)^{3}}.
$

If $\varphi(0)<\infty$, then $\varphi$ is non-strict generator or generator of the non-strict Archimedean copula 
$\displaystyle
C(u,v)=\varphi^{-1}(\min{(\varphi(u)+\varphi(v),\varphi(0))}).
$
 A non-strict generator can always be standardized. The standard generator satisfies $\varphi(0)=1$. Thus, all generators of the same Archimedean copula are scalar multiples of the standard generator. So, without loss of generality we can state all results in terms of the standard generator. A stationary Markov chain can be defined by a copula and a one dimensional marginal distribution. In this set-up, we call it a copula-based Markov chain. 
For stationary Markov chains with uniform marginals, the transition probability for sets $A=(-\infty, y]$ is $P(x,A)=C_{,1}(x,y)$. See Longla and Peligrad (2012) for more details on this topic.  

\subsection{Mixing coefficients}
 Given $\sigma$-fields $\mathscr{A},\mathscr{B}$, the mixing coefficients of interest in this paper are defined as follows.

\noindent $\displaystyle \beta(\mathscr{A},\mathscr{B})=\mathbb{E}\sup_{B\in \mathscr{B}}|P(B|\mathscr{A})-P(B)|,$  \quad
$\displaystyle \rho(\mathscr{A},\mathscr{B})=\sup_{f\in \mathbb{L}^{2}(\mathscr{A}),g\in \mathbb{L}^{2}(\mathscr{B})}corr(f,g),$ 
and

\noindent 
$\displaystyle \phi(\mathscr{A},\mathscr{B})=\sup_{B\in \mathscr{B}, A\in \mathscr{A}, P(A)>0}|P(B|A)-P(B)|.$
Using the transition probabilities for a Markov chain generated by an absolutely continuous copula and a marginal distribution with strictly positive density, for  $\mathscr{A}=\sigma(X_{i}, i\leq 0)$,$\mathscr{B}=\sigma(X_{i}, i\geq n)$, it was shown by Longla and Peligrad (2012) that 
$$ \beta(\mathscr{A}, \mathscr{B})=\beta_{n}=\int_{0}^{1}\sup_{B\in\mathcal{R}\cap I}|\int_{B}(c_{n}(x,y)-1)dy|dx,$$
$$\rho(\mathscr{A}, \mathscr{B})=\rho_{n}=\sup_{f,g}\big\{\int^{1}_{0}\int^{1}_{0}c_{n}(x,y)f(x)g(y)dxdy : \quad ||g||_{2}=||f||_{2}=1, \quad \mathbb{E}(f)=\mathbb{E}(g)=0\big\},$$
$\displaystyle \phi(\mathscr{A}, \mathscr{B})=\phi_{n}=\sup_{B\in \mathcal{R}\cap I}ess\sup_{x}|\int_{B}(c_{n}(x,y)-1)dy|,$
where, $c_n$ is the density of $(X_{0}, X_{n})$.  

A stochastic process is $\rho$-mixing, if $\rho_{n}\rightarrow 0$. The process is exponentially mixing, if the convergence rate is exponential. A stochastic process is geometrically ergodic, if $\beta_{n}$ converges to $0$ exponentially fast.  A stationary sequence is absolutely regular, if
$\beta_{n}\rightarrow0$ as $n\rightarrow\infty.$ It is well known (see for
instance Corollary 21.7 in Bradley  vol. 2 (2007)) that a strictly stationary Markov chain is absolutely regular (i.e. $\beta_{n}\rightarrow0$), if and only if it is irreducible (i.e. Harris recurrent) and aperiodic. A stationary Markov chain is irreducible if there exists a set $B$, such that $\pi(B)=1$ and the following holds: for all $x\in B$ and every set $A\in\mathcal{R}$ such that $\pi(A)>0$, there is a positive integer $n=n(x,A)$ for which $P^{n}(x,A)>0.$ An irreducible stationary Markov chain is aperiodic if and only if there is $A$ with $\pi(A)>0$, $n>0$, such that $P^{n}(x,A)>0$ and $P^{n+1}(x,A)>0$ for all $x$ $\in A$. Here $\pi$ is the invariant distribution. See  Theorem $3.3.1$ of Chan and Tong (2001) for more. Let $(Y_{n}, n\in\mathbb{N})$ be an irreducible and aperiodic Markov chain with a transition measure $P^{n}(x, A) = P(X_{n}\in A|X_{0} = x)$, $n \geq 1$. S is a small set, if it is nonnull and for some $n>0$, $q >0$ and a probability measure $\nu$, such that $P^{n}(x, A) \geq q\nu(A)$ for all $x \in S$ and measurable $A$ .

This paper is structured as follows. In Section 2 we provide new results on exponential $\rho$-mixing and exponential $\beta$-mixing for some families of copulas. For instance, Lemma \ref{Lem1} and Theorem \ref{Theo1} deal with mixing rates of copula-based Markov chains with square-integrable copula densities. Theorem \ref{theo5} provides a new bound on $\rho_{1}$ and relates it to our previous results. Theorem \ref{Theo9} generalizes the result of Theorem \ref{Theo2}. Theorem \ref{Theo3} is about mixing rates of Markov chains generated by non-strict Archimedean copulas. In Proposition \ref{cor2} we study convex combinations of copulas. In Theorem \ref{MHMixing} we provide the mixing rates of a popular kernel, and the subsequent Lemma \ref{MHcopula} exhibits a new family of copulas and the mixing rate of the Markov chains they generate. Families of copulas that generate $\rho$-mixing and $\phi$-mixing are constructed. We also apply the theory of small sets to the Frechet and Mardia copula families. We show, that Markov chains generated by these families of copulas are $\phi$-mixing, thus geometric $\beta$-mixing and $\rho$-mixing. In Section 3 we provide the proofs.
\section{Mixing rates of copula-based Markov chains}
 A copula-based Markov chain is the representation of a stationary Markov chain by the copula of its consecutive states and an invariant distribution. 
\subsection{General condition for exponential $\rho$-mixing}
Define the linear operator $T: \mathbb{L}^{2}_{0}(0,1)\to \mathbb{L}^{2}(0,1)$ by \begin{eqnarray}
T(f)(x)=\int_{0}^{1}{f(y)c(x,y)dy}. \label{Tf}
\end{eqnarray}
It is well known that $\displaystyle \rho_{1}=\sup_{f\in\mathbb{L}^{2}_{0}(0,1)}\frac{||Tf||_{2}}{||f||_{2}}$. See Longla and Peligrad (2012) for references. Based on this fact, we derive the following.
\begin{lemma}\label{Lem1}
For a stationary Markov chain generated by a symmetric copula with square-integrable density, $\rho_{k}=\lambda_{1}^{k}$. $\rho$-mixing is equivalent to $\lambda_{1}<1$, where $\lambda_{1}$ is the largest eigen-value of  $T$. 
\end{lemma}
Combining this result with Theorem 4 of Longla and Peligrad (2012) leads to the following theorem.
\begin{theorem} \label{Theo1}
A stationary Markov chain generated by a symmetric copula with square integrable density is geometric $\beta$-mixing if and only if it is geometric $\rho$-mixing. 
\end{theorem}

 Beare (2010) has shown that geometric ergodicity follows from $\rho$-mixing and the reverse implication uses the comment before Theorem 4 in Longla and Peligrad (2012).

\begin{theorem} \label{theo5}  {
 
\begin{eqnarray}
 \mbox{Let}\quad c(x,1)-c(x,0)\in \mathbb{L}^{2}(0,1) \label{cond} \quad \mbox{and}\quad
\int_{0}^{1}|c_{y}(x,y)|dy \in \mathbb{L}^{2}(0,1).\\
\mbox{Define} \quad ||\int_{0}^{1}|c_{y}(x,y)|dy||^{2}_{2}=k_{1} \quad \mbox{and} \quad || |c(x,1)-c(x,0)|+\int_{0}^{1}|c_{y}(x,y)|dy||^{2}_{2}=k_{2}.
\end{eqnarray}
If $k_{1}+k_{2}<12$, then the stationary Markov chain generated by $C$ is an exponential $\rho$-mixing ($\rho_{1}\le \sqrt{(k_{1}+k_{2})/12}<1$). Moreover, if the density is strictly positive on a set of Lebesgue measure 1, then it is geometrically ergodic.
}
\end{theorem}
\begin{exam}{The Farlie-Gumbel-Morgenstern family of copulas.}
The Farlie-Gumbel-Morgenstern copula family defined by $\displaystyle C(x,y)=xy+\theta xy(1-x)(1-y)$, $\theta\in[-1,1]$, generates exponential $\rho$-mixing and exponential $\beta$-mixing for all values of its parameter. 
\end{exam}
For this family, $\displaystyle c(x,y)=1+\theta(1-2x)(1-2y), \hskip2mm c_{y}(x,y)=-2\theta(1-2x)$. All assumptions of Theorem \ref{theo5} are satisfied. 
$\displaystyle c(x,1)-c(x,0)=-2\theta(1-2x), \quad \int_{0}^{1}|c_{y}(x,y)|dy=2|\theta(1-2x)| $ and
$\displaystyle |c(x,1)-c(x,0)|+\int_{0}^{1}|c_{y}(x,y)|dy=4|\theta(1-2x)|.$
Therefore, $k_{1}=4\theta^2/3$ and $k_{2}=16\theta^2/3$.

\noindent $k_{1}+k_{2}<12$ if $\theta^{2} <9/5$. This is true for all $\theta \in [-1,1].$

Beare  (2010) proved that for a copula with density bounded away from zero we have exponential $\rho$-mixing. These conditions imply $\phi$-mixing as shown by Longla and Peligrad. These assumptions were relaxed by Longla (2013) as follows.

\begin{theorem} \label{Theo2}
If there exists nonnegative functions $\varepsilon_{1},$ $\varepsilon_{2}$ defined on $[0,1]$, for which the density of the absolute continuous part of the copula satisfies the inequality $\displaystyle c(x,y)\geq \varepsilon_{1}(x)+\varepsilon_{2}(y),$ with $\varepsilon_{1}$, $\varepsilon_{2} \in \mathbb{L}^{1}[0,1]$, such that at least one of the two functions has a non-zero integral, then the Markov chains generated by this copula are exponential $\rho$-mixing. Moreover, if the density is strictly positive on a set of Lebesgue measure 1, then these Markov chains are geometrically ergodic.
\end{theorem}
\begin{remark} This theorem improves Theorem 4.2 of Beare (2010), by extending it to cases when the density can actually be equal to zero on a set of non-zero measure, and therefore not be bounded away from 0.  In the example below, we exhibit a copula that provides exponential $\rho$-mixing, but was ruled out by Theorem 4.2 of Beare (2010).
\end{remark}

\begin{exam}
Given any bounded and integrable functions $h: [0,1] \rightarrow [0,1],$  $g: [0,1]\rightarrow [0, 1] $, let $b_{1}=\sup{g}$, $a_{1}=\inf{g}$, $b_{2}=\sup{h}$ and $a_{2}=\inf{h}$. 

\begin{table}[th]
\vline
\centering
\rowcolors{1}{white}{gray!35}
\begin{tabular}{|c|}
\hline\hline
 Copula densities  \\[.5ex]
\hline
 $m_{1}(x,y)=\frac{b_{1}-g(x)h(y)+h(y)||g||_1 +g(x)||h||_1}{b_{1}+||g||_1||h||_1}$ \\[1ex]

\hline
 $m_{2}(x,y)=\frac{b_{1}b_{2}-g(x)h(y)+h(y)||g||_1 +g(x)||h||_1}{b_{1}b_{2}+||g||_1||h||_1}$  \\[1ex]

\hline
$m_{3}(x,y)=\frac{b_{1}(b_{2}-a_{2})-g(x)(b_{2}-h(y))+(b_{2}-h(y))||g||_1 +g(x)(b_{2}-||h||_1)}{b_{1}(b_{2}-a_{2})+||g||_1 (b_{2}-||h||_1)}$  \\[1ex]

\hline  $m_{4}(x,y)=\frac{(b_{1}-a_{1})(b_{2}-a_{2})-(b_{1}-g(x))(b_{2}-h(y))+(b_{2}-h(y))(b_{1}-||g||_1) +(b_{1}-g(x))(b_{2}-||h||_1)}{(b_{1}-a_{1})(b_{2}-a_{2})+(b_{1}-||g||_1) (b_{2}-||h||_1)}$ \\
\hline
\end{tabular}\vline

\caption{New copula families}
\label{tab: See1}
\end{table}
\end{exam}
Functions defined in Table \ref{tab: See1} are densities of copulas that generate exponential $\rho$-mixing Markov chains. The respective maximal correlation coefficients are bounded as shown in the Table \ref{tab: See2}.

\begin{table}[th]
\vline
\centering
\rowcolors{1}{white}{gray!35}
\begin{tabular}{|l|}
\hline\hline
 Upper bound on $\rho_{1}$ for the new families  \\[.5ex]
\hline
 $\rho_{1}(m_{1})\leq \frac{b_{1}}{b_{1}+||g||_1 ||h||_1}<1$ \\[1ex]

\hline
 $\rho_{1}(m_{2})\leq \frac{b_{1}b_{2}}{b_{1}b_{2}+||g||_1 ||h||_1}< 1$  \\[1ex]

\hline
$\rho_{1}(m_{3})\leq \frac{b_{1}(b_{2}-a_{2})}{b_{1}(b_{2}-a_{2})+||g||_1 (b_{2}-||h||_1)}<1$  \\[1ex]

\hline  $\rho_{1}(m_{4})\leq \frac{(b_{1}-a_{1})(b_{2}-a_{2})}{(b_{1}-a_{1})(b_{2}-a_{2})+(b_{1}-||g||_1) (b_{2}-||h||_1)}<1$ \\
\hline
\end{tabular}\vline

\caption{Upper bound on $\rho_{1}$ for the new copula families}
\label{tab: See2}
\end{table}

The copula densities in Table \ref{tab: See3} generate $\phi$-mixing Markov chains. 

\begin{table}[th]
\vline
\centering
\rowcolors{1}{white}{gray!35}
\begin{tabular}{| l| r |}
\hline\hline
 copula density   & parameters of the family  \\[.5ex]
\hline
 $c(x,y)=\frac{\frac{3}{2^{2-a}}+1+(1/2-y)x^{1/a-1}sign(1/2-x^{1/a})}{1+\frac{3}{2^{2-a}}}$ & $a\in(0,1]$ \\[1ex]

\hline
 $c(x,y)=1+\frac{\theta}{2a}x^{1/a-1}(2y-1)sign(1/2-x^{1/a}) $ & $\theta \in [-2a, 2a], \quad a\in(0,1]
$  \\[1ex]

\hline
$c(x,y)=\frac{c+1+\frac{\theta}{2a}x^{1/a-1}(2y-1)sign(1/2-x^{1/a})}{1+c} $  & $\theta \in [-2a, 2a], \quad a\in(0,1], \quad c\in\mathbb{R}^{+}$  \\[1ex]

\hline  $c(x,y)=\frac{c+1+ (1/2-y)x^{1/a-1}sign(1/2-x^{1/a})}{1+c}$  & $a\in(0,1], \quad c\in\mathbb{R}^{+}$ \\
\hline
\end{tabular}\vline

\caption{New $\phi$-mixing copula families}
\label{tab: See3}
\end{table}

The first entry of Table \ref{tab: See3} is bounded away from $0$ for all values of the parameter $a\in(0,1]$. The second entry of the table is bounded away from $0$ when $|\theta|< 2a$. The third entry of this table generates exponential $\phi$-mixing stationary Markov chains for all $c>0$ and $\theta\leq 2a$. Concerning the last entry, for all $c>0$ and $0<a\leq1$, the density is bounded away from $0$. Thus, it generates $\phi$-mixing Markov chains. 

The proof of Theorem \ref{Theo2}, unveals a more general result.

\begin{theorem}\label{Theo9} Let $f(x,y)$ be a nonnegative function in $\mathbb{L}^{1}(0,1)$ satisfying the following properties: 
\begin{enumerate}
\item $\int_{I}f(x,y)dx=1$ a.s;
\item $\int_{I}f(x,y)dy=1$ a.s.
\end{enumerate}
If there exist nonnegative functions $\varepsilon_{1}$ and $\varepsilon_{2}$ in $\mathbb{L}^{1}(0,1)$ such that  $f(x,y)\geq \varepsilon_{1}(x)+\varepsilon_{2}(y)$ a.s, then  $\displaystyle\int_{I}\varepsilon_{1}(x)dx+\int_{I}\varepsilon_{2}(x)dx < 2 \quad \mbox{and}$ $$ \Big|\int_{I^{2}}f(x,y)g(x)h(y)dxdy\Big|\leq \Big(1-\frac{1}{2}\Big(\int_{I}\varepsilon_{1}(x)dx+\int_{I}\varepsilon_{2}(x)dx
\Big)\Big)\Big(\int_{I}g^{2}(x)dx\Big)^{1/2}\Big(\int_{I}h^{2}(x)dx\Big)^{1/2}.$$
\end{theorem}

\subsection{Exponential $\rho$-mixing for Archimedean copulas}
 
Archimedian copulas have been studied by many researchers and are very popular. Beare (2012) proved that under some mild conditions, some strict Archimedian copulas generate geometrically ergodic Markov chains. Longla and Peligrad (2012) have shown that those assumptions imply $\rho$-mixing. We provide here a new result for non-strict Archimedean copulas.

\begin{theorem} \label{Theo3}
\quad

Let $\varphi$ be a non-strict standard generator of an Archimedean copula not equal to the Hoeffding lower bound. Assume $\varphi$ has a second derivative. The copula generates exponential $\rho$-mixing Markov chains if 
$\displaystyle \int_{0}^{1}(1-x)\Big(\frac{h(x)}{(\varphi^{'}o\varphi^{-1}(x))^{2}}\Big)^{2}dx<1, $ 
where $\displaystyle h(x)=\max_{0\leq y\leq 1-x}\varphi^{''}o\varphi^{-1}(x+y).$
\end{theorem} 
Notice that, if $\varphi^{''}$ is decreasing, then $h(x)=\varphi^{''}(0)$,  and if $\varphi^{''}$ is increasing, then $h(x)=\varphi^{''}o\varphi^{-1}(x)$. Moreover, if $\varphi^{'}(1)\neq0$, then it is enough to have $\displaystyle \int_{0}^{1}h^{2}(x)(1-x)dx<(\varphi^{'}(1))^4.
$

\begin{exam}
The Archimedean copula with generator $\tilde{\varphi}(u)=-\ln(\theta u+1-\theta)$, $\theta \in (0,1)$. 
\end{exam}
The standard generator is $\displaystyle\varphi(x)=\frac{\ln(\theta u+1-\theta)}{\ln(1-\theta)}$ and 
$\displaystyle \varphi^{-1}(x)=\frac{(1-\theta)^{x}-1+\theta}{\theta}$. $$\varphi^{'}(x)=\frac{1}{\ln(1-\theta)(x+\frac{1-\theta}{\theta})}, \quad \varphi^{'}o\varphi^{-1}(x)=\frac{\theta}{\ln(1-\theta)(1-\theta)^{x}}. $$
So, $\displaystyle \varphi^{''}o\varphi^{-1}(x+y)= \frac{\theta^{2}(1-\theta)^{-2(x+y)}}{-\ln(1-\theta)}, \quad h(x)=\frac{\theta^{2}(1-\theta)^{-2}}{-\ln(1-\theta)}.$ Therefore, we need

$$\int_{0}^{1}(1-x)\Big(\frac{h(x)}{(\varphi^{'}o\varphi^{-1}(x))^{2}}\Big)^{2}dx=
-\frac{\ln(1-\theta)}{4(1-\theta)^{4}}+\frac{1}{16}-\frac{1}{16(1-\theta)^{4}}<1.$$
Thus, copulas from this family generate exponential $\rho$-mixing Markov chains for $ \theta\in(0,\theta_{0}),$ where $\theta_{0} \approxeq .348$ is the unique value of $\theta$ for which the inequality becomes an equality. 

\begin{remark}
This example is taken from the list of Archimedean copulas  in \S{6} of Nelsen (2006). On this example we can see that Theorem \ref{Theo3} doesn't handle the case of $P$, corresponding to $\theta=1$ because this copula is strict. The case $\theta=0$ is ruled out by the assumptions. 
\end{remark}
 \begin{exam}
 The non-strict generator $\displaystyle
\varphi(x)=\frac{1-x}{1+(\theta-1)x}=\varphi^{-1}(x), \quad \theta\in [1,\infty).
$
\end{exam}
$$ \varphi^{'}(x)= \frac{-\theta}{(1+(\theta-1)x)^{2}}, \quad \varphi^{'}o\varphi^{-1}(x)=\frac{(1+(\theta-1)x)^{2}}{-\theta}, \quad \varphi^{''}(x)=\frac{-2\theta(\theta-1)}{(1+(\theta-1)x)^{3}}.$$
For this family of copulas,  $\varphi$ is decreasing. Therefore, $\displaystyle h(x)=-2\theta(\theta-1)$.

\noindent $\displaystyle \rho_{1}^{2}\leq \int_{0}^{1}\Big(\frac{(1-x)^{1/2}h(x)}{(\varphi^{'}o\varphi^{-1}(x))^{2}}\Big)^{2}dx =\frac{2}{21}+\frac{4}{7}\theta^{7}-\frac{2}{3}\theta^{6}=f(\theta).$
$f(\theta)$ is an increasing function on $[1,\infty)$ with $f(1)=0$, $f(\infty)=\infty$. Thus, there exists a unique $\theta_{0}\approxeq 1.388$ for which $f(\theta_{0})=1$. Therefore, the copula generates exponential $\rho$-mixing Markov chains for $\theta \in (1, \theta_{0})$.

\subsection{Convex combinations of copulas}
We will use in this section various methods to assess the rate of convergence of mixing coefficients of Markov chains generated by some copula families. We will use direct computation for $\rho$-mixing. We will also show how small sets are related to geometric ergodicity, and apply Theorem 8 of Longla and Peligrad (2012) to the example of the Mardia and Frechet families of copulas. Longla and Peligrad (2012) have shown the following.

\begin{lemma} \label{lemma1}
Any convex combination of geometrically ergodic reversible Markov chains is geometrically ergodic.
\end{lemma}
$\rho$-mixing and absolute regularity imply geometric ergodicity. Thus, Lemma \ref{lemma1} implies the folowing:

\begin{proposition} \label{cor2}
{ The Markov chain generated by any convex combination of copulas, one of which generates $\rho$-mixing Markov chains and another one generates absolutely regular Markov chains, is geometrically ergodic and exponential $\rho$-mixing.}
\end{proposition}

\begin{exam}{Exponential $\beta$-mixing for the Frechet and Mardia families of copulas}
\begin{equation}
C(x,y)=\frac{\theta^{2}(1+\theta)}{2}M(x,y)+(1-\theta^2)P(x,y)+\frac{\theta^{2}(1-\theta)}{2}W(x,y), \label{Mardia} \quad \theta\in[-1,1],
\end{equation} 
\begin{equation}
C(x,y)=C_{a,b}(x,y)=a M(x,y) +(1-a-b)P(x,y)+b W(x,y) \quad (0\leq a+b\leq 1). \label{Frechet}
\end{equation}
\end{exam}
(\ref{Mardia}) defines the Mardia family of copulas and (\ref{Frechet}) defines the Frechet family of copulas. Notice that a Mardia copula is a Frechet copula with $a+b=\theta^{2}$. Any copula from this family has a singular part (See Longla (2013) for more). 
We shall show the following.

\begin{theorem} \label{theo4}
{ A stationary Markov chain generated by a copula from the Frechet (Mardia) family with uniform marginal has $n$-steps joint cumulative distribution function $\displaystyle C_{a_{n},b_{n}}(x,y)$, where 
\begin{equation}
a_{n}=\frac{1}{2}[(a+b)^{n}+(a-b)^{n}], \quad b_{n}=\frac{1}{2}[(a+b)^{n}-(a-b)^{n}].
\end{equation}
The Markov chain is exponentially $\phi$-mixing, therefore $\rho$-mixing and geometrically $\beta$-mixing for $a+b\ne 1$. For copulas with $a+b= 1$, there is no mixing.}
\end{theorem}

\begin{remark}
\quad

\begin{enumerate}
\item First of all, notice that Theorem \ref{Theo2} can be applied to both families to show that we have exponential $\rho$-mixing for $a+b\neq 1$. Theorem \ref{Theo2} can't be used for copulas with $a+b=1$. 
\item The absolute continuous part of the copula for these families has density $1-a-b$. So, we can conclude, using Theorem 8 of Longla and Peligrad (2012), that this density being bounded away from zero when $1-a-b\neq 0$, the copula families generate $\phi$-mixing. This, on its own, implies  geometric ergodicity and $\rho$-mixing for the Markov chains generated by these copulas.
\end{enumerate}
\end{remark}

\subsection{Practical Example for simulation studies}
 A popular kernel that is used to generate Markov chains with a given probability of staying at the same state $x$ equal to $p(x)$, where $x\in [-1, 1]$, is defined by:
$$Q(x, A)=p(x)\delta_{A}(x)+(1-p(x))\nu(A), \quad \mbox{where $\nu$ is a probability measure on} \quad [-1, 1].$$
If $\displaystyle \theta=\int^{1}_{-1}\frac{\nu(dx)}{1-p(x)}< \infty$, then the invariant distribution is defined by $\displaystyle\pi(dx)=\frac{\nu(dx)}{\theta(1-p(x))}.$ For references on this example, see Longla and Peligrad (2012). If we allow the acceptance probability to depend on a parameter $a$, then require the marginal distribution to be uniform on $[-1, 1]$ and $\nu$ absolutely continuous with respect to the Lebesgue measure, having density $h(x,a)$, then it follows that
$\displaystyle Q(x,A)=p(x,a)\delta_{A}(x)+(1-p(x, a))\nu(A)$.
$\pi(dx)=\frac{1}{2}dx$ implies $\displaystyle h(x,a)=k(1-p(x,a)), \quad\mbox{where}\quad \theta=2k.$
To analyze the mixing structure of the Markov chain generated by this transition kernel and the given invariant distribution, we derive the corresponding copula.
The corresponding copula representation of the transition probability $\mathbb{P}(x,(-1,y])$ is given by $\displaystyle C_{,1}(\frac{x+1}{2},\frac{y+1}{2})=p(x,a)\mathbb{I}(x\leq y)+k(1-p(x,a))\int_{-1}^{y}(1-p(t,a))dt.$
So, using the transformation formula $(U,V)=(F(X),F(Y))$ and the Sklar's theorem, it follows that 
$\displaystyle C_{,1}(u,v)=p(2u-1,a)\mathbb{I}(u\leq v)+k(1-p(2u-1,a))\int_{-1}^{2v-1}(1-p(t,a))dt.$  Using the notation $f(x)=\int_{-1}^{x}p(t,a)dt$ and integrating with respect to $u$, we obtain
the copula 
\begin{eqnarray} C(u,v)=\frac{1}{2}\Big[f\Big(\min(2u-1,2v-1)\Big)+k(2u-f(2u-1))(2v-f(2v-1))\Big]. \label{copMH} 
\end{eqnarray}
If we take $p(x,a)= a|x|$ with $a \leq1$, then 
$\displaystyle Q(x,A)=a|x|\delta_{A}(x)+k(1-a|x|)\int_{A}(1-a|t|)dt.$

\begin{theorem} \label{MHMixing}
The stationary Markov chain generated by the above transition kernel and the uniform distribution is exponential $\rho$-mixing and geometrically ergodic for $a<1$. The Markov chain is $\beta$-mixing with rate $1/n$, but not $\rho$-mixing when $a=1$.
\end{theorem}

\begin{lemma} \label{MHcopula}
Any function of the form (\ref{copMH}) with an increasing differentiable function $f$ satisfying $f(-1)=0$ and $f(1)=2-1/k$ defines a one parameter copula family for $2\geq k > 0$. This family generates exponential $\rho$-mixing Markov chains for $0< k<2$.
\end{lemma}

\section{Appendix: Mathematical proofs}
\subsection{Lemma \ref{Lem1}}
The first part of the concluion belongs to Beare (2010). Assume (U,V) is a random vector with distribution $C(u,v)$. The square integrable density of $C$ defines a Hilbert-Schmidt operator, and therefore a compact operator. For an operator with these properties, there exists a basis of eigen-functions in $\mathbb{L}^{2}(0,1)$.  Reversibility implies a spectral representation of the kernel of $T$ in the form 
$\displaystyle c(u,v)=\sum_{i=0}^{\infty} \lambda_{i}\varphi_{i}(u)\varphi_{i}(v),$
where $\varphi_{i}(u)$ are the eigen-functions corresponding to the eigen-values $\lambda_{i}$ of $T$, and form an orthonormal basis of $\mathbb{L}^{2}(0,1)$.  $\lambda_{i}\in\mathbb{R}^{+}$ is a decreasing sequence. $\lambda_{0}=1$ is eigen-value with eigen-function $1$. Therefore, $\lambda^{k}_{i}$ are eigen-values of the operator $T^{k}$ with kernel $c^{k}$ corresponding to the same eigen-functions. So, 
$\displaystyle c^{k}(u,v)=1+\sum_{i=1}^{\infty} \lambda^{k}_{i}\varphi_{i}(u)\varphi_{i}(v).
$
 $\rho_{k}$ coefficient becomes
$\displaystyle
\rho_k=\sup_{f,g}\Big\{\int_{0}^{1}\int_{0}^{1}c^{k}(u,v)f(u)g(v)dudv : \mathbb{E}(f)=\mathbb{E}(g)=0, \mathbb{E}(f^{2})=\mathbb{E}(g^{2})=1\Big\}.
$

 Bounding this quantity by use of Jensen's inequality, then H\"{o}lder's inequality, we obtain

\noindent$\displaystyle |\int_{0}^{1}\int_{0}^{1}c^{k}(u,v)f(u)g(v)dudv|=|\sum_{i=1}^{\infty}
\lambda^{k}_{i}\int_{0}^{1}\varphi_{i}(u)f(u)du\int_{0}^{1}\varphi_{i}(v)g(v)dv|\leq $

$\displaystyle \leq \sum_{i=1}^{\infty}\lambda^{k}_{i}\Big(\int_{0}^{1}\varphi^{2}_{i}(u)du\Big)^{1/2}\Big(\int_{0}^{1}f^{2}(u)du\Big)^{1/2}
\Big(\int_{0}^{1}\varphi^{2}_{i}(v)dv\Big)^{1/2}\Big(\int_{0}^{1}g^{2}(v)dv\Big)^{1/2}\leq \sum_{i=1}^{\infty}\lambda^{k}_{i}.
$

\noindent So, for $k\geq 2$, we have 
$\displaystyle
\rho_{k}\leq \sum_{i=1}^{\infty} \lambda^{k}_{i}\leq \lambda^{k-2}_{1}\sum_{i=1}^{\infty}\lambda^{2}_{i}\leq M\lambda^{k}_{1}. $
Here $\displaystyle M=(\sum_{i=1}^{\infty}\lambda^{2}_{i})/\lambda^{2}_{1}$. The series converges because we have a Hilbert-Schmidt operator.
 Therefore, if $\lambda_{1} <1$, then $\rho_{k}$ converges to $0$ exponentially fast. On the other hand,  
$\displaystyle c^{k}(u,v)=1+\lambda_{1}^{k}\varphi_{1}(u)\varphi_{1}(v)+\sum_{i=2}^{\infty} \lambda^{k}_{i}\varphi_{i}(u)\varphi_{i}(v).$
Because the basis is orthonormal, we have $\displaystyle corr(\varphi_{1}(U), \varphi_{1}(V))=\int_{0}^{1}\int_{0}^{1}c^{k}(u,v)\varphi_{1}(u)\varphi_{1}(v)dudv$. Therefore, $\displaystyle corr(\varphi_{1}(U), \varphi_{1}(V))=\lambda^{k}_{1}\int_{0}^{1}\varphi^{2}_{1}(u)du\int_{0}^{1}\varphi_{1}^{2}(v)dv=\lambda^{k}_{1}.$
Thus, $1\geq\rho_{k}\geq \lambda^{k}_{1}$. Therefore, if $\lambda_{1}=1$, then $\rho_{k}=1$ for all $k$. So, we have exponential $\rho$-mixing if and only if $\rho_{1}<1$. By formula (5) in Longla and Peligrad (2012), $\displaystyle\rho_{1}=\sup_{f\in L^{2}_{0}(0,1)}\frac{||Tf||_{2}}{||f||_{2}}$. In this case, this norm is $\lambda_{1}$. Thus, $\rho_{1}=\lambda_{1}$. Also, it is well known that $\rho_{k}\leq \rho_{1}^{k}$, and we have just shown, that $\lambda_{1}^{k}\leq\rho_{k}$. Thus, $\rho_{k}=\lambda_{1}^{k}$.

\subsection{Theorem \ref{Theo1}}
$\displaystyle \beta_{k}=\int_{0}^{1}\sup_{B}|\int_{B}(c(u,v)-1)du|dv.$ 
Using Jensen's and H\"{o}lder's inequalities leads to

\noindent$\displaystyle
\beta_{k}=\int_{0}^{1}\sup_{B}|\int_{B}(\sum_{i=1}^{\infty} \lambda^{k}_{i}\varphi_{i}(u)\varphi_{i}(v))du|dv \leq \sum_{i=1}^{\infty} \lambda^{k}_{i}\int_{0}^{1}|\varphi_{i}|(v)dv\sup_{B}\int_{B}|\varphi_{i}|(u)du. 
$

\noindent$\displaystyle
\mbox{So, } \beta_{k} \leq \sum_{i=1}^{\infty} \lambda^{k}_{i}\int_{0}^{1}|\varphi_{i}|(v)dv\sup_{B}\mu^{1/2}(B)\Big(\int_{B}\varphi^{2}_{i}(u)du\Big)^{1/2}.$
Using  $0\leq \mu(B)\leq 1,$ $B\subset[0,1]\cap\mathcal{R}$, where $\mu$ is the Lebesgue measure, we obtain
$\displaystyle
\beta_{k} \leq \sum_{i=1}^{\infty} \lambda^{k}_{i}\int_{0}^{1}|\varphi_{i}|(v)dv.$
Thus, H\"{o}lder's inequality implies $\displaystyle \beta_{k}\leq \sum_{i=1}^{\infty} \lambda^{k}_{i}\Big(\int_{0}^{1}\varphi^{2}_{i}(v)dv\Big)^{1/2}=\sum_{i=1}^{\infty} \lambda^{k}_{i}.$
Therefore, $\beta_{k}\leq M\lambda^{k}_{1}$. So, $\beta_{k}$ converges exponentially to $0$ when $\rho_{1}<1$. Now, if we have geometric ergodicity, then Theorem 4 of Longla and Peligrad (2012) holds. Therefore, the Markov chain is $\rho$-mixing.

\subsection{Theorem \ref{Theo2}}
  Let $f$, $g$ be two functions with $||f||_{2}=||g||_{2}=1, \quad\mathbb{E}(f(X))=\mathbb{E}(g(Y))=0$, where $X$ and $Y$ have uniform distributions on $[0,1]$. We have 
\begin{equation}
2f(x)g(y)=f^{2}(x)+g^{2}(y)-(f(x)-g(y))^{2}. \label{Al}
\end{equation}
Therefore,
$$2\int_{I^2}f(x)g(y)C(dx,dy)=\int_{I^2}f^{2}(x)C(dx,dy)+\int_{I^2}g^{2}(y)C(dx,dy)-\int_{I^2}(f(x)-g(y))^{2}C(dx,dy).$$
Using the fact that $\displaystyle\int_{I}C(dx,dy)=dx$ and $\displaystyle\int_{I}f^{2}(x)dx=\int_{I}g^{2}(x)dx=1$, we obtain
$$ \int_{I^2}f^{2}(x)C(dx,dy)=\int_{I}f^{2}(x)\int_{I}C(dx,dy)=\int_{I}f^{2}(x)dx=1=\int_{I^2}g^{2}(y)C(dx,dy).$$ 
On the other hand, using $c(x,y)\geq \varepsilon_{1}(x)+\varepsilon_{2}(y)$ on a set of Lebesgue measure 1,
$$\int_{I^2}(f(x)-g(y))^{2}C(dx,dy)\geq \int_{I^2}(f(x)-g(y))^{2}(\varepsilon_{1}(x)+\varepsilon_{2}(y))dxdy=$$
$$=\int_{I^2}(f^{2}(x)+g^{2}(y)-2f(x)g(y))(\varepsilon_{1}(x)+\varepsilon_{2}(y))dxdy= I_{a}+I_{b},$$
where
$\displaystyle I_{b}=\int_{I^2}f^{2}(x)\varepsilon_{2}(y)dxdy+\int_{I^2}g^{2}(y)\varepsilon_{2}(y)dxdy-2\int_{I^2}f(x)g(y)\varepsilon_{2}(y)dxdy
$

\noindent and $\displaystyle I_{a}=\int_{I}f^{2}(x)\varepsilon_{1}(x)dx\int_{I}dy+\int_{I}\varepsilon_{1}(x)dx\int_{I}g^{2}(y)dy
-2\int_{I}f(x)\varepsilon_{1}(x)dx\int_{I}g(y)dy.$
The cross terms are equal to zero. Moreover, $\displaystyle\int_{I}g^{2}(x)dx=1$ and $\displaystyle\int_{I}f^{2}(x)\varepsilon_{1}(x)dx\geq 0$. Whence,

\noindent $\displaystyle I_{a}= \int_{I}f^{2}(x)\varepsilon_{1}(x)dx+\int_{I}\varepsilon_{1}(x)dx\geq \int_{I}\varepsilon_{1}(x)dx.$
Similarly, $\displaystyle I_{b}\geq \int_{I}\varepsilon_{2}(y)dy$. 

\noindent Thus, 
$\displaystyle -\int_{I^2}(f(x)-g(y))^{2}C(dx,dy)\leq -\int_{I}\varepsilon_{1}(x)dx-\int_{I}\varepsilon_{2}(y)dy.$
Using (\ref{Al}) and integrating, we obtain
$\displaystyle 2\int_{I^2}f(x)g(y)C(dx,dy)\leq 2-(\int_{I}\varepsilon_{1}(x)dx+\int_{I}\varepsilon_{2}(y)dy).$

\noindent It follows that 
$\displaystyle
corr(f,g)\leq 1-\frac{1}{2}(\int_{I}\varepsilon_{1}(x)dx+\int_{I}\varepsilon_{2}(y)dy). $
Because this holds for all such $f$ and $g$, it also holds for $f$ and $-g$. Thus, 
$\displaystyle
\sup_{f,g}|corr(f,g)|\leq 1-\frac{1}{2}(\int_{I}\varepsilon_{1}(x)dx+\int_{I}\varepsilon_{2}(y)dy).
$ 
Provided that one of the two integrals of the right hand side is non-zero (say $\varepsilon$). It follows that 
$\displaystyle\rho_{1} \leq 1-\frac{1}{2}\varepsilon < 1.
$
If, in addition, the density is positive on a set of Lebesgue measure 1, then the assumptions of Theorem 4 of Longla and Peligrad (2012) hold. Thus, geometric ergodicity follows.

\subsection{Theorem \ref{Theo3}}

Let $X=\varphi(U)$ and $Y=\varphi(V)$, where $(U,V)$ has distribution $C(u,v)$. The common marginal distribution function and probability distribution function of $X$ and $Y$ are
\begin{equation}
\mathbb{P}(X\leq x)=\mathbb{P}(Y\leq x)=1-\varphi^{-1}(x), \quad f_{X}(x)=f_{Y}(x)=\frac{1}{-\varphi^{'}o\varphi^{-1}(x)}.\label{DensX}
\end{equation}
Using this transformation and $\mathbb{P}(U\geq u, V\geq v)=-1+\mathbb{P}(U\geq u)+\mathbb{P}(V\geq v)+\mathbb{P}(U<u,V<v)$ yields the joint cumulative  distribution function of $(X,Y)$
$$\mathbb{P}(X\leq x,Y\leq y)=\mathbb{P}(U\geq \varphi^{-1}(x), V\geq \varphi^{-1}(y))=1-\varphi(x) -\varphi(y) +C(\varphi^{-1}(x), \varphi^{-1}(y)).$$
Differentiating this function gives the joint density
$\displaystyle
h(x,y)=-\frac{\varphi^{''}o\varphi^{-1}(x+y)}{\Big(\varphi^{'}o\varphi^{-1}(x+y)\Big)^{3}}\mathbb{I}\{x+y\leq 1 \}.
$

\noindent$\displaystyle corr(f(U),g(V))=corr(f o \varphi^{-1} (\varphi(U)), g o \varphi^{-1} (\varphi(V)))=corr(\tilde{f}(X), \tilde{g}(Y))$, with $\tilde{f}=fo\varphi^{-1}$ and $\tilde{g}=go\varphi^{-1}$. The function $\varphi^{-1}$ is defined on $[0,1]$ because we are using the standard generator. So, $\varphi^{-1}o\varphi(X)=X$. Therefore, $\rho_{1}(\sigma(X),\sigma(Y))=\rho_{1}(\sigma(U),\sigma(V))$. 

Let $f$ and $g$ be such that $\mathbb{E}(f(X))=\mathbb{E}(g(X))=0$, $Var(f(X))=Var(g(X))=1$. Given the formula of the density, the correlation coefficient between $f(X)$ and $g(Y)$ can be computed by
$\displaystyle corr(f(X),g(Y))=\int_{0}^{1}\int_{0}^{1-x}\frac{\varphi^{''}o\varphi^{-1}(x+y)f(x)g(y)}
{\Big(-\varphi^{'}o\varphi^{-1}(x+y)\Big)^{3}}dydx.$ Therefore, 

\noindent $\displaystyle |corr(f(X),g(Y))|\leq \int_{0}^{1}|f(x)|\int_{0}^{1-x}\frac{\varphi^{''}o\varphi^{-1}(x+y)|g(y)|}
{\Big(-\varphi^{'}o\varphi^{-1}(x+y)\Big)^{3}}dydx.
$
 Recall, that $\varphi^{-1}$ is decreasing because $\varphi$ is decreasing. Thus,  $\displaystyle \forall x,y\in [0,1],  0=\varphi^{-1}(1)\leq \varphi^{-1}(x+y)\leq \varphi^{-1}(x)\leq \varphi^{-1}(0)=1.$ So, 
$\displaystyle\varphi^{'}(0)=\varphi^{'}o\varphi^{-1}(1)\leq \varphi^{'}o\varphi^{-1}(x+y)\leq \varphi^{'}o\varphi^{-1}(x)\leq \varphi^{'}o\varphi^{-1}(0)=\varphi^{'}(1)\leq 0$ because  $\varphi$ is convex. 

\noindent
$\displaystyle 0\leq \frac{1}{\Big(-\varphi^{'}(0)\Big)^3}\leq \frac{1}{\Big(-\varphi^{'}o\varphi^{-1}(x+y)\Big)^3}\leq \frac{1}{\Big(-\varphi^{'}o\varphi^{-1}(x)\Big)^{5/2}\Big(-\varphi^{'}o\varphi^{-1}(y)\Big)^{1/2}}\leq 
\frac{1}{\Big(-\varphi^{'}(1)\Big)^3}.
$
Therefore, using (\ref{DensX}) and $\displaystyle h(x)=\max_{0\leq y\leq 1-x}\varphi^{''}o\varphi^{-1}(x+y)$ leads to

\begin{eqnarray}
|corr(f(U),g(V))|\leq 
\int_{0}^{1}\frac{h(x)|f(x)|}{\Big(-\varphi^{'}o\varphi^{-1}(x)\Big)^{5/2}}\int_{0}^{1-x}\frac{|g(y)|dy}{\Big(-\varphi^{'}o\varphi^{-1}(y)\Big)^{1/2}}dx.
\label{corr1} 
\end{eqnarray}
Using twice H\"{o}lder's inequality in  (\ref{corr1}) yields 

$\displaystyle
|corr(f(U),g(V))|\leq 
\int_{0}^{1}\frac{h(x)|f(x)|(1-x)^{1/2}}{\Big(-\varphi^{'}o\varphi^{-1}(x)\Big)^{5/2}}dx \leq \Big(\int_{0}^{1}\Big(\frac{h(x)}{(-\varphi^{'}o\varphi^{-1}(x))^{2}}\Big)^{2}(1-x)dx\Big)^{1/2}. 
$
Therefore, taking the supremum over all such $f$ and $g$,
$\displaystyle
\rho_{1}^{2}\leq \int_{0}^{1}\Big(\frac{h(x)}{(\varphi^{'}o\varphi^{-1}(x))^{2}}\Big)^{2}(1-x)dx . $

\noindent So, $\rho_{1}< 1$, if $\displaystyle\int_{0}^{1}\Big(\frac{h(x)}{(\varphi^{'}o\varphi^{-1}(x))^{2}}\Big)^{2}(1-x)dx<1 \quad \mbox{or} \quad \int_{0}^{1}h^{2}(x)(1-x)dx<(\varphi^{'}(1))^4.$

\subsection{Theorem \ref{theo4}}
Let $(Y_{n}, n\in\mathbb{N})$ be a Markov chain generated by $C_{a,b}$ and the uniform distribution on $I$. The formula of the $n$-steps transition copula is based on the multiplicative property of the copula families and the recurrence relationship that can be easily established between copulas of $(Y_{0}, Y_{n})$ and $(Y_{0}, Y_{n+1})$. 

For the proof of Doeblin recurrence, we need to show by Theorem 8 of Longla and  Peligrad (2012), that the density c(u,v) of the absolutely continuous part of the copula is bounded away from 0 on a set of Lebesgue measure 1. This follows from the fact that $c(u,v)>(1-a-b)$ for all $u,v \in [0,1]$. Therefore, the Markov chain these copulas generate are $\phi$-mixing for $a+b\neq 1$. This implies geometric ergodicity and exponential $\rho$-mixing by the corollary to Theorem 8 of  Longla and  Peligrad (2012).

Theorem B.1.4 in Chan and Ton (2001) states the following.

\begin{theorem} \label{Lya}
\quad

Let $(Y_{n}, n\in\mathbb{N})$ be an irreducible and aperiodic Markov chain. Suppose there exists a small set $S$, a nonnegative measurable function $L$, which is bounded away from $0$ and $\infty$ on $S$, and constants $r >1$, $\gamma >0$, $K >0$, such that
$\displaystyle r\mathbb{E}(L(X_{n+1})|X_{n}=x) \leq L(x) -\gamma, \quad \mbox{for all}\quad x\in S^{c}, \quad \mbox{and}$

$\displaystyle\int_{S^{c}}L(w)P(x, dw) < K, \quad \mbox{for all} \quad x \in S.$
Then, $(X_{n}, n\in\mathbb{N})$ is geometrically ergodic. 
\end{theorem}
Here, $L$ is called the Lyapunov function. We shall use this result as follows.
We will use small sets to show that the Markov chain above is geometrically ergodic, then apply Theorem 4 Longla and Peligrad (2012) to obtain $\rho$-mixing.
Assume $a+b\neq 1$. Let $S=[1/2,1]$. For any $A\in\mathcal{R}$, $P(x,A)$ is a sum of three components, one of which is $(1-a-b)\mu(A)$. So,
$\displaystyle P(x,A) \ge (1-a-b)\mu(A) \quad \mbox{for all} \quad x\in S.$
Taking $q=1-a-b$, $\nu=\mu$ and $n=1$, we conclude that $S$ is a small set.

Now, we shall show that the Markov chain generated by this copula is irreducible and aperiodic, and there exists a Lyapunov function.
The density of the absolute continuous part of the copula is $c(u,v)\ge 1-a-b>0$. This density being strictly positive on a set of Lebesgue measure 1, we can conclude by Proposition 2 of M. Longla and M. Peligrad (2012), that the stationary Markov chain it generates is absolutely regular, and thus irreducible and aperiodic.

We shall now return to the existence of the Lyapunov function for geometric ergodicity of the Markov chains generated by these copulas.
Notice that, if $x\in S$, then $1-x\in S^{c}$. Therefore, for any function $L\in \mathbb{L}^{1}(0,1)$, $\displaystyle \mathbb{E}(L(X_{1})|X_{0}=x)=aL(x)+bL(1-x)+(1-a-b)\int_{0}^{1}L(x)dx.$

 Let $L(x)=\mathbb{I}(x\ge 1/2)+2\mathbb{I}(x<1/2).$ Therefore, $\displaystyle\int_{0}^{1}L(x)dx=3/2.$ So,  for $x\in S^{c}$, $L(x)=2$, $L(1-x)=1$ and
$\displaystyle\mathbb{E}(L(X_{1})|X_{0}=x)=2a+b+(1-a-b)(3/2)=2(\frac{a+3}{4})-\frac{b}{2},$
leading to $\frac{4}{a+3}\mathbb{E}(L(X_{1})|X_{0}=x)=L(x)-\frac{b}{a+3}$ for all $x\in S^{c}$. So, $r=\frac{4}{a+3}>1$ and $\gamma=\frac{b}{a+3}>0$. 

On the other hand, $L$ being bounded on $I$, the second condition holds. Therefore, $L$ is a Lyapunov function for the Markov chain generated by this copula. So, by Theorem \ref{Lya}, this stationary Markov chain is geometrically ergodic for $a+b\neq 1$.

 For any convex combination of the two copulas $M$ and $W$ ( corresponding to $a+b=1$), we can compute $\rho_{1}$ as follows. The corresponding transition operator acts on functions in $\mathbb{L}^{2}(0,1)$ via $Qf(u)=a f(u)+(1-a)f(1-u).$
Therefore, if we can find a function $f$ defined on $I$, such that $\mathbb{E}(f)=0$ and $f(1-u)=f(u)$ for all $u$, then $Qf(u)=f(u)$. The existence of such a function leads to $\rho_{1}=1$. The function $f(x)=cos(2\pi x)$ satisfies these assumptions. In conclusion, the Markov chains generated by the copulas are not $\rho$-mixing, and due to symmetry, they are not geometrically ergodic and not Doeblin recurrent. 

\subsection{ Theorem \ref{theo5}}
Consider $T: \mathbb{L}^{2}_{0}(0,1) \rightarrow \mathbb{L}^{2}(0,1)$, 
$\displaystyle T(f)(x)=\int_{0}^{1}{f(y)c(x,y)dy}.
$
We shall use the following claim.
\begin{claim} \label{CL}
 Let $\mathbb{H}$ be a Hilbert space. Let $T$ be a bounded operator defined on $\mathbb{H}$, and $\{e_{n}(x), n\in \mathbb{N}\}$ be an orthonormal basis of $\mathbb{H}$. Then, $\displaystyle
||T||^2 \leq \sum_{n\ge1} ||T(e_{n})||_{2}^2.
$ 

\end{claim}

\begin{proof}
$\displaystyle f(x)=\sum_{n\ge 1}{a_{n}}e_{n}(x) \quad \mbox{implies}\quad ||f||_{2}=(\sum_{n\ge1}{a^{2}_{n}})^{1/2}.$ Also,  

\noindent $\displaystyle ||Tf||_{2}=||(\sum_{n\ge1}{a_{n}}Te_{n})||_{2} \leq (\sum_{n\ge1}{a^{2}_{n}})^{1/2}(\sum_{n\ge1}||Te_{n}||_{2}^{2})^{1/2}$ implies $\displaystyle \frac{||Tf||_{2}}{||f||_{2}}\leq (\sum_{n\ge1}||Te_{n}||_{2}^{2})^{1/2}$. 
The last inequality uses Cauchy's inequality. This leads to
$\displaystyle ||T||=\sup_{f}\frac{||Tf||_{2}}{||f||_{2}}\leq (\sum_{n\ge1}||Te_{n}||_{2}^{2})^{1/2}$. 
\end{proof}

 It remains to estimate $||Te_{n}||_{2}^{2}$  for the most convenient orthonormal basis of $\mathbb{L}^{2}_{0}(0,1)$, 

\noindent$\displaystyle
\{ e_{n}=\sqrt{2}\sin(2n{\pi}x) ,  \quad b_{n}=\sqrt{2}\cos(2n{\pi}x), \quad n\geq 1\}. 
$
\vskip2mm
{\bf Case 1. $e_{n}=\sqrt{2}\cos(2n{\pi}x)$}. 
$$(1/\sqrt{2})Te_{n}(x)=\int_{0}^{1}c(x,y)\cos(2{\pi}ny)dy=
c(x,y)\frac{\sin(2n{\pi}y)}{2{\pi}n}|_{y=0}^{1}-\frac{1}{2n\pi}\int_{0}^{1}c_{y}(x,y)\sin(2{\pi}ny)dy.$$

Therefore, using $|\sin(2\pi ny)|\leq 1$ and (\ref{cond}), we obtain
$\displaystyle
|(1/\sqrt{2})Te_{n}(x)| \leq \frac{1}{2n\pi}\int_{0}^{1}|c_{y}(x,y)|dy. $ So,
$\displaystyle ||Te_{n}||_{2}^{2}\leq\frac{1}{2({n\pi})^{2}}||\int_{0}^{1}|c_{y}(x,y)|dy||^{2}_{2}.$
In our notations, we obtain
$\displaystyle ||Te_{n}||_{2}^{2}\leq\frac{k_{1}}{2({n\pi})^{2}}.$

\vskip2mm
{\bf Case 2. $e_{n}=\sqrt{2}\sin(2n{\pi}x)$.}
Same as above with the only difference that for this case, the first part is not zero, but $-(\frac{1}{2n\pi})[c(x,1)-c(x,0)]$. 

So,
 $\displaystyle
|(1/\sqrt{2})Te_{n}(x)| \leq \frac{1}{2n\pi}[|c(x,1)-c(x,0)|+\int_{0}^{1}|c_{y}(x,y)|dy]. $
Computing the norms yields
$\displaystyle ||Te_{n}||_{2}^{2}\leq ||\frac{1}{\sqrt{2}n\pi}[|c(x,1)-c(x,0)|+\int_{0}^{1}|c_{y}(x,y)|dy]||^{2}_{2}.$
In our notations,
$\displaystyle
||Te_{n}||_{2}^{2} \leq \frac{k_{2}}{2(n\pi)^2}.
$

\noindent Taking into account both cases, using $\displaystyle \sum_{n>0}\frac{1}{n^{2}}=\frac{\pi^{2}}{6},$  \quad
$\displaystyle
\sum_{n>1} ||T(e_{n})||_{2}^2=\frac{k_{1}+k_{2}}{2(\pi)^{2}}\sum_{i>0}\frac{1}{i^{2}}=\frac{k_{1}+k_{2}}{12}.
$
So, $k_{1}+k_{2} < 12$ implies $\rho_{1}=||T||<1$. Therefore, the generated Markov chains are exponential $\rho$-mixing.
Moreover, if the density is stricly positive on a set of Lebesgue measure 1, then $\beta$-mixing follows from $\rho$-mixing by Theorem 4 of Longla and Peligrad (2012).
\subsection{Theorem \ref{MHMixing}}
To assess the rate of convergence of the mixing coefficient $\beta_{n}$, we use Lemma 2 of Doukhan and al. (1994), that can be stated as follows for the transition kernel at hands. 

\begin{lemma} \label{DMR}
For the algorithm of interest in Theorem \ref{MHMixing}, the following holds.
$$\mathbb{E}_{\pi}(p^{n}(X,a)) \leq \beta_{n}\leq 3\mathbb{E}_{\pi}(p^{[n/2]}(X,a)).$$
\end{lemma}

Applying  Lemma \ref{DMR} and computing the necessary expected values lead to $$\frac{a^{n+1}}{n+1}\leq \beta_{n}\leq 3\mathbb{E}_{\pi}(p^{[n/2]}(X,a))=\frac{3a^{[n/2]+1}}{[n/2]+1}\leq C\rho^{n}, \quad \rho=\sqrt{a}.$$
Therefore, the Markov chain is exponential $\beta$-mixing when $a<1$. Reversibility implies exponential $\rho$-mixing by Theorem 4 of Longla and Peligrad (2012).  For $a=1$, the generated Markov chain is a $\beta$-mixing with decay rate $1/n$, but fails to be $\rho$-mixing. It fails to be $\rho$-mixing because, by Theorem 4 of Longla and Peligrad (2012), it would have been otherwise geometrically ergodic.

\section{Acknowledgments}
The author thanks his advisor Dr Magda Peligrad for her support in life, for her questions and for her help in the preparation of this paper.


\begin{thebibliography}{100}

\bibitem {Beare1}  Beare, B.K., (2012).  Archimedean copulas and temporal dependence. Econometric Theory, vol. 28, 1165--1185.

\bibitem{Beare} Beare, B.K., (2010). Copulas and Temporal Dependence.  Econometrica, vol. 78, 395–-410.

\bibitem{Brad} Bradley, R. C., (2007). Introduction to strong mixing conditions. Vol. 1,2. Kendrick press.

\bibitem{ChanTon} Chan, K. and Tong, H., (2001). Chaos: A Statistical Perspective. Springer, New York.

\bibitem{Doukhan} Doukhan, P., Massart, P. and  Rio, E. (1994). The functional central limit theorem for strongly mixing processes. Annales de l'I.H.P., section B, Tome 30, \textnumero 1, 63--82.

\bibitem{Martial1} Longla, M. and Peligrad, M. (2012). Some Aspects of Modeling Dependence in Copula-based Markov chains. Journal of Multivariate Analysis, vol. 111, 234--240.

\bibitem{Martial2} Longla, M (2013). Remarks on the speed of convergence of mixing coefficients and applications. Statistics and Probability Letters, DOI: 10.1016/j.spl.2013.04.002.

\bibitem{Nelsen} Nelsen, R.B., (2006). An introduction to copulas. 2nd ed. Springer, New York.
\end{thebibliography}
\end{document}